\documentclass[10pt, a4paper]{amsart}
\usepackage{geometry} 
\usepackage{url}
\usepackage{palatino}
\usepackage{amsthm}
\usepackage{amssymb}
\usepackage{latexsym}
\usepackage{amsfonts}
\usepackage{amsmath}
\usepackage{amstext}
\usepackage{accents}
\usepackage{cancel}
\usepackage{color}
\usepackage{soul}
\usepackage{enumerate}
\usepackage{eucal}
\usepackage{amscd}
\usepackage{hyperref}
\usepackage{graphicx}
\usepackage{verbatim}
\usepackage{caption}
\usepackage{enumitem}

\usepackage[usenames]{xcolor}
\setstcolor{red}

\definecolor{greenish}{RGB}{50,160,0}

\newtheorem{theorem}{Theorem}[section]

\newtheorem{lemma}[theorem]{Lemma}

\newtheorem{corollary}[theorem]{Corollary}

\theoremstyle{definition}
\newtheorem{definition}[theorem]{Definition}
\newtheorem{example}[theorem]{Example}

\theoremstyle{remark}
\newtheorem{remark}[theorem]{Remark}

\numberwithin{equation}{section}

\begin{document}
\title[Minimal surfaces in ${\mathbb{R}}^{4}$ and holomorphic null curves in ${\mathbb{C}}^{4}$]{Minimal surfaces in ${\mathbb{R}}^{4}$ foliated by conic sections and \\ parabolic rotations of holomorphic null curves in ${\mathbb{C}}^{4}$}

\author[Hojoo Lee]{Hojoo Lee}
\address{Center for Mathematical Challenges, Korea Institute for Advanced Study, Hoegiro 85, Dongdaemun-gu, Seoul 02455, Korea}
\email{momentmaplee@gmail.com}

\begin{abstract}
Using the complex parabolic rotations of holomorphic null curves in ${\mathbb{C}}^{4}$, we transform minimal surfaces in 
Euclidean space ${\mathbb{R}}^{3} \subset {\mathbb{R}}^{4}$ to a family of degenerate minimal surfaces in Euclidean space ${\mathbb{R}}^{4}$. 
Applying our deformation to holomorphic null curves in ${\mathbb{C}}^{3} \subset {\mathbb{C}}^{4}$ induced by helicoids in ${\mathbb{R}}^{3}$, we discover new minimal surfaces in ${\mathbb{R}}^{4}$ foliated by conic sections with eccentricity grater than $1$: hyperbolas or straight lines. Applying our deformation to  holomorphic null curves in ${\mathbb{C}}^{3}$ induced by catenoids in ${\mathbb{R}}^{3}$, we can rediscover the Hoffman-Osserman  
catenoids in ${\mathbb{R}}^{4}$ foliated by conic sections with eccentricity smaller than $1$: ellipses or circles. We prove the
existence of minimal surfaces in ${\mathbb{R}}^{4}$ foliated by ellipses, which converge to circles at infinity. We construct minimal 
surfaces in ${\mathbb{R}}^{4}$ foliated by parabolas: conic sections which have eccentricity $1$.
\end{abstract}

% \subjclass{53A10, 49Q05}
 
 \maketitle
   
  \begin{center}
  {\small{\textit{To the memory of Professor Ahmad El Soufi}}}
  \end{center}
 
  \bigskip

 \section{Introduction}
  
  It is a fascinating fact that any simply connected minimal surface in ${\mathbb{R}}^{3}$ admits a $1$-parameter family of minimal isometric deformations,  called the associate family. Rotating the holomorphic null curve in ${\mathbb{C}}^{3}$ induced by the given  minimal surface in ${\mathbb{R}}^{3}$ realizes such isometric deformation of minimal surfaces. For instance, catenoids (foliated by circles) belong to the associate families of helicoids (foliated by lines) in ${\mathbb{R}}^{3}$. 

R. Schoen \cite{Sch83} characterized catenoids as the only complete, embedded minimal surfaces  in  ${\mathbb{R}}^{3}$ with finite topology and two ends. F. L\'{o}pez  and A. Ros \cite{LR91}
 used the so called L\'{o}pez-Ros deformation (See Example \ref{LRdef}) of holomorphic null curves in ${\mathbb{C}}^{3}$ to prove that planes and catenoids are the only embedded complete minimal surfaces in ${\mathbb{R}}^{3}$ with finite total curvature and genus zero.  
The L\'{o}pez-Ros deformation has various interesting geometric applications \cite{BB2011, LR91, MP2012, PR1993, PR2002, Ros1996}. Meeks and Rosenberg \cite{MeRo05} proved that helicoids are  the unique complete, simply connected, properly embedded, nonplanar minimal surface in 
${\mathbb{R}}^{3}$ with one end. 

I. Castro and F. Urbano \cite{CU1999} established the uniqueness of $n$-dimensional Lagrangian catenoid \cite[Section III. 3. B.]{HL1982} in ${\mathbb{R}}^{2n}$. In particular, in ${\mathbb{R}}^{4}$, the Lagrangian catenoid can be identified as the holomorphic  curve  in ${\mathbb{C}}^{2}$:  $\left\{ \left({\zeta}, \frac{\lambda}{\zeta} \right) \in {\mathbb{C}}^{2}  \; \vert \;   \zeta \in  {\mathbb{C}} - \{0\}   \right\}$ for a  constant $\lambda \in  {\mathbb{C}} - \{0\}$.
Lagrangian catenoids are the only non-planar special Lagrangian surfaces foliated by circles. S.-H. Park \cite{Park2015} showed that 
a non-planar circle-foliated minimal surface in ${\mathbb{R}}^{n \geq 3}$ should be a part of Lagrangian catenoid in ${\mathbb{R}}^{4}$ or 
 classical minimal surfaces in ${\mathbb{R}}^{3}$: catenoids and Riemann's minimal surface \cite{MP2015, Riem1868} foliated by circles or lines. 

R. Osserman \cite[Theorem 9.4]{Oss86} established that catenoids and Enneper's surfaces are the only complete minimal 
surfaces in  ${\mathbb{R}}^{3}$ with  total curvature $-4 \pi$.  The Lagrangian catenoids in ${\mathbb{R}}^{4}$ have total 
curvature $-4 \pi$.  More generally, D. Hoffman and R. Osserman  \cite[Chapter 6]{HO80} classified 
all complete minimal surfaces in ${\mathbb{R}}^{n \geq 3}$ with total curvature $-4 \pi$. The Hoffman-Osserman list includes a family of minimal surfaces  in ${\mathbb{R}}^{4}$ foliated by ellipses or circles.   

 We first review the classical deformations of minimal surfaces in ${\mathbb{R}}^{n \geq 3}$ induced by deformations of holomorphic null curves 
in ${\mathbb{C}}^{n}$.  We use the complex parabolic rotations (Lemma \ref{deform 1})  of holomorphic null curves  in ${\mathbb{C}}^{4}$ 
to show that the deformation of classical catenoids in  ${\mathbb{R}}^{3}$  to the Hoffman-Osserman catenoids in  ${\mathbb{R}}^{4}$ can be 
generalized to deformations of arbitrary minimal surfaces in ${\mathbb{R}}^{3}$ to a family of degenerate minimal surfaces 
in ${\mathbb{R}}^{4}$ (Theorem \ref{thm:main}). As an application of our deformation, generalizing helicoids in ${\mathbb{R}}^{3}$ foliated by 
lines, we construct new minimal surfaces in ${\mathbb{R}}^{4}$ foliated by conic sections with eccentricity grater than $1$: hyperbolas or lines
 (Example \ref{deformation of helicoids}).  We establish the existence of minimal surfaces in ${\mathbb{R}}^{4}$ foliated by ellipses, which 
 converges to circles at infinity (Theorem \ref{twocircles}).  Finally, we also provide examples of minimal surfaces in ${\mathbb{R}}^{4}$ 
 foliated by conic sections with eccentricity $1$: parabolas  (Example \ref{complex parabola}).

 \section{Deformations of minimal surfaces   in ${\mathbb{R}}^{n}$ and holomorphic curves in ${\mathbb{C}}^{n}$}

 Since two dimensional minimal surfaces in real Euclidean space is parametrized by conformal harmonic mappings, they can be constructed from 
  holomorphic null curves in complex Euclidean space. Let's review classical deformations of simply connected minimal surfaces 
  in ${\mathbb{R}}^{n}$ induced by deformations of the holomorphic null curves in ${\mathbb{C}}^{n}$. 

 \begin{example}[\textbf{Associate family of minimal surfaces in ${\mathbb{R}}^{3}$}]
Let  ${\mathbf{X}}:\Sigma \rightarrow {\mathbb{R}}^{3}$ denote a conformal harmonic immersion of a simply connected minimal surfaces $\Sigma$ induced by the holomorphic null curve $ \left({\phi}_{1}, {\phi}_{2}, {\phi}_{3} \right)$:
\[
  {\mathbf{X}}(\zeta) =   {\mathbf{X}}({\zeta}_{0}) + \left(  \textrm{Re} \int_{{\zeta}_{0}}^{\zeta}  {\phi}_{1}(\zeta)  d\zeta,  \textrm{Re} \int_{{\zeta}_{0}}^{\zeta}  {\phi}_{2}(\zeta)  d\zeta,  \textrm{Re}  \int_{{\zeta}_{0}}^{\zeta}   {\phi}_{3}(\zeta)  d\zeta \right), 
\]
or concisely,
\[
  {\mathbf{X}}(\zeta) = \left(  \textrm{Re} \int  {\phi}_{1}(\zeta)  d\zeta,  \textrm{Re} \int  {\phi}_{2}(\zeta)  d\zeta,  \textrm{Re} \int  {\phi}_{3}(\zeta)  d\zeta \right).
\]
 Given a real angle constant $\theta$, we rotate the initial holomorphic null curve $ \left({\phi}_{1}, {\phi}_{2}, {\phi}_{3} \right)$ to associate 
 by the new holomorphic null curve $ \left({\widetilde{\phi}}_{1},   {\widetilde{\phi}}_{2}, {\widetilde{\phi}}_{3} \right)$:
 \[
 \begin{bmatrix} 
  {\phi}_{1} \\  {\phi}_{2} \\ {\phi}_{3}
 \end{bmatrix} 
 \mapsto 
  \begin{bmatrix} 
   {\widetilde{\phi}}_{1} \\   {\widetilde{\phi}}_{2} \\  {\widetilde{\phi}}_{3}
 \end{bmatrix} 
 =
 \begin{bmatrix} 
e^{-i \theta}  & 0 &  0 \\ 
0 & e^{-i \theta} &  0 \\ 
 0 & 0 & e^{-i \theta}
 \end{bmatrix} 
  \begin{bmatrix} 
 {\phi}_{1} \\  {\phi}_{2} \\ {\phi}_{3}
 \end{bmatrix}.
 \]
The holomorphic null curve $ \left({\widetilde{\phi}}_{1},   {\widetilde{\phi}}_{2}, {\widetilde{\phi}}_{3} \right)$ induces the minimal surface  ${\Sigma}^{\theta}$ is given by the conformal harmonic immersion 
\[
  {\mathbf{X}}^{\theta}(\zeta) =  \left(  \textrm{Re} \int  {\widetilde{\phi}}_{1}(\zeta)  d\zeta,   \textrm{Re} \int  {\widetilde{\phi}}_{2}(\zeta)  d\zeta,  \textrm{Re} \int  {\widetilde{\phi}}_{3}(\zeta)  d\zeta   \right).
\]
The minimal surface ${\Sigma}^{\frac{\pi}{2}}$ is called the conjugate surface of ${\Sigma}^{0}= {\Sigma}$. Under the conjugate transformation ${\Sigma} \mapsto {\Sigma}^{\frac{\pi}{2}}$, the lines of curvature on $\Sigma$ map to the asymptotic lines on ${\Sigma}^{\frac{\pi}{2}}$ and the asymptotic lines on $\Sigma$ map to the lines of curvature on ${\Sigma}^{\frac{\pi}{2}}$.
\end{example}

 \begin{example}[\textbf{Lawson's isometric deformation  \cite{Law71} of minimal surfaces in ${\mathbb{R}}^{3}$ to minimal surfaces in ${\mathbb{R}}^{6}$}]
 Given a simply connected minimal surface $\Sigma$ in ${\mathbb{R}}^{3}$, we construct Lawson's two-parameter family of minimal 
 surfaces in ${\mathbb{R}}^{6}$, which are isometric to $\Sigma$. Suppose that the minimal surface  $\Sigma$ is obtained by integrating
  the holomorphic null curve $ \left({\phi}_{1},   {\phi}_{2}, {\phi}_{3} \right)$ in ${\mathbb{C}}^{3}$. First,  given a real angle constant $\beta$, we obtain
 the minimal surface ${\Sigma}_{\beta}$ in ${\mathbb{R}}^{6}$ induced by the holomorphic null curve in ${\mathbb{C}}^{6}$: 
 \[
 \left({\varphi}_{1}, \cdots, {\varphi}_{6} \right) =\cos  {\beta}  \left({\phi}_{1}, 0,  {\phi}_{2},0, {\phi}_{3}, 0 \right)
 + \sin  \beta  \left(0, -i {\phi}_{1},  0, -i {\phi}_{2}, 0, -i {\phi}_{3} \right).
 \]
We notice that the minimal surface ${\Sigma}_{\frac{\pi}{4}}$ in ${\mathbb{R}}^{6}$ can be identified with the holomorphic 
curve  $\frac{1}{\sqrt{2}} \left({\phi}_{1},   {\phi}_{2}, {\phi}_{3} \right)$ in ${\mathbb{C}}^{3}$. 
Second, given another real angle constant $\alpha$, we rotate the  holomorphic null curve $ \left({\varphi}_{1}, \cdots, {\varphi}_{6} \right)$ to get
the holomorphic null curve 
\[
 \left( \widetilde{{\varphi}_{1}}, \cdots,  \widetilde{{\varphi}_{6}} \right) = e^{-i \alpha}  \left({\varphi}_{1}, \cdots, {\varphi}_{6} \right).
\]
Integrating this, we obtain the minimal 
surface ${\Sigma}_{\left(\alpha, \beta \right)}$ in ${\mathbb{R}}^{6}$, which is isometric to ${\Sigma}$ in ${\mathbb{R}}^{3}$. 
Lawson \cite[Theorem 1]{Law71} proved  that this is the only way to deform a minimal surface ${\Sigma}$ in ${\mathbb{R}}^{3}$
 isometrically to minimal surfaces in  Euclidean space ${\mathbb{R}}^{n \geq 3}$. 
 \end{example}
  
\begin{example}[\textbf{Goursat's transformation of minimal surfaces in ${\mathbb{R}}^{3}$ \cite{Deu2013, DHS2010, G1887, LK2016}}] \label{goursat}
The Goursat deformation transforms minimal surfaces in ${\mathbb{R}}^{3}$ to minimal surfaces in ${\mathbb{R}}^{3}$.
Let  ${\mathbf{X}}(\zeta):\Sigma \rightarrow {\mathbb{R}}^{3}$ be the immersion of a simply connected minimal surfaces $\Sigma$ induced by the holomorphic null curve $ \left({\phi}_{1},   {\phi}_{2}, {\phi}_{3} \right)$.
 
For each $t \in \mathbb{R}$, we associate the holomorphic null curve $ \left({\widetilde{\phi}}_{1},   {\widetilde{\phi}}_{2}, {\widetilde{\phi}}_{3} \right)$ by the linear map:
\[
 \begin{bmatrix} 
  {\phi}_{1} \\  {\phi}_{2} \\ {\phi}_{3}
 \end{bmatrix} 
 \mapsto 
  \begin{bmatrix} 
   {\widetilde{\phi}}_{1} \\   {\widetilde{\phi}}_{2} \\  {\widetilde{\phi}}_{3}
 \end{bmatrix} 
 =
 \begin{bmatrix} 
\cosh t &- i \sinh t   &  0 \\ 
 i \sinh t & \cosh t &  0 \\ 
 0 & 0 & 1 
 \end{bmatrix} 
  \begin{bmatrix} 
 {\phi}_{1} \\  {\phi}_{2} \\ {\phi}_{3}
 \end{bmatrix}.
 \]
This induces the conformal harmonic map of the minimal surface 
 \[
  {\mathbf{X}}^{t}(\zeta) = \left(  \textrm{Re} \int  {\widetilde{\phi}}_{1}(\zeta)  d\zeta,   \textrm{Re} \int  {\widetilde{\phi}}_{2}(\zeta)  d\zeta,  \textrm{Re} \int  {\widetilde{\phi}}_{3}(\zeta)  d\zeta   \right).
\]
The nullity of the holomorphic curve $ \left({\widetilde{\phi}}_{1},   {\widetilde{\phi}}_{2}, {\widetilde{\phi}}_{3} \right)$ comes from 
 the identity, for $t \in \mathbb{R}$,  
\[ 
{z_1}^2 +{z_2}^2  =  {\widetilde{z_1}\,}^{2} + {\widetilde{z_2}\,}^{2},  \; \begin{bmatrix} 
 \widetilde{z_1}   \\   \widetilde{z_2} 
 \end{bmatrix} 
 = \begin{bmatrix} 
\cosh t & - i \sinh t  \\   i \sinh t & \cosh t 
 \end{bmatrix}
 \begin{bmatrix} 
 z_1  \\   z_2  
 \end{bmatrix}
= \begin{bmatrix} 
\cos \left( i t \right) & - \sin  \left( i t \right)  \\     \sin  \left( i t \right) & \cos  \left( i t \right) 
 \end{bmatrix}
 \begin{bmatrix} 
 z_1  \\   z_2  
 \end{bmatrix}.
\]
See also \cite[Section 2]{Romon1997} by P. Romon.
\end{example}

\begin{example}[\textbf{L\'{o}pez-Ros deformation  \cite{LR91} of minimal surfaces in ${\mathbb{R}}^{3}$}]   \label{LRdef}
In terms of Weierstrass datum of holomorphic null curves, we rewrite Goursat's deformation more geometrically. (For instance, see \cite[Section 2.1.1]{PR2002}.) Let $\Sigma$ be a simply connected minimal surface in ${\mathbb{R}}^{3}$, up to translations, parametrized by the conformal harmonic map
  \begin{equation}
 \left( {\mathbf{x}}_{1}(\zeta), {\mathbf{x}}_{2}(\zeta), {\mathbf{x}}_{3}(\zeta)  \right) = \left(  \textrm{Re} \int  {\phi}_{1}(\zeta)  d\zeta,   \textrm{Re} \int  {\phi}_{2}(\zeta)  d\zeta,  \textrm{Re} \int  {\phi}_{3}(\zeta)  d\zeta \right),
  \end{equation}
  where the curve $  \left(     {\phi}_{1}, \,   {\phi}_{2}, \,   {\phi}_{3}\right) $ is determined by the Weierstrass data $(G(\zeta), \Psi(\zeta) d\zeta):$ 
\[
  \left(     {\phi}_{1}, \,   {\phi}_{2}, \,   {\phi}_{3}\right) = 
   \left(  \frac{1}{2}  \left(1 -G^2 \right)  \Psi, \,  \frac{i}{2} \left(1 +G^2 \right)  \Psi, \,   G  \Psi  \right).
\]
Geometrically, the meromorphic function $G$ is the complexified Gauss map under the stereographic projection of the induced unit normal on the minimal surface. Given a constant $\lambda>0$, taking $t= - \ln \lambda$ in Goursat transfromation (in Example \ref{goursat}), we have 
  the linear deformation of holomorphic null curves:
\[
 \begin{bmatrix} 
  {\phi}_{1} \\  {\phi}_{2} \\ {\phi}_{3}
 \end{bmatrix} 
 \mapsto 
  \begin{bmatrix} 
   {\widetilde{\phi}}_{1} \\   {\widetilde{\phi}}_{2} \\  {\widetilde{\phi}}_{3}
 \end{bmatrix} 
 =
 \begin{bmatrix} 
  \frac{1}{2}  \left(   \frac{1}{\lambda} +  \lambda   \right) & -   \frac{i}{2}  \left(   \lambda -  \frac{1}{\lambda}   \right)  &  0 \\ 
\frac{i}{2}  \left(   \lambda -  \frac{1}{\lambda}   \right)  &  \frac{1}{2}  \left(   \frac{1}{\lambda} +  \lambda   \right) &  0 \\ 
 0 & 0 & 1 
 \end{bmatrix} 
  \begin{bmatrix} 
 {\phi}_{1} \\  {\phi}_{2} \\ {\phi}_{3}
 \end{bmatrix}.
 \]
This induces the conformal harmonic map of the minimal surface ${\Sigma}^{\lambda}:$ 
 \[
 \left( {\mathbf{x}}^{\lambda}_{1}(\zeta), {\mathbf{x}}^{\lambda}_{2}(\zeta), {\mathbf{x}}^{\lambda}_{3}(\zeta)  \right) = 
 \left(  \textrm{Re} \int  {\phi}^{\lambda}_{1}(\zeta)  d\zeta,   \textrm{Re} \int  {\phi}^{\lambda}_{2}(\zeta)  d\zeta,  \textrm{Re} \int  {\phi}^{\lambda}_{3}(\zeta)  d\zeta   \right).
\]
The Goursat deformation of holomorphic null curves yields the deformation of the Weierstrass datum of the so called L\'{o}pez-Ros deformation:
\[
 (G(\zeta), \Psi(\zeta) d\zeta) \mapsto \left(\, {G}^{\lambda}(\zeta), \, {\Psi}^{\lambda}(\zeta) d\zeta \right) = \left( \lambda G(\zeta), \frac{1}{\lambda} \Psi(\zeta) d\zeta \right). 
\]
Under the L\'{o}pez-Ros deformation ${\Sigma} \mapsto {\Sigma}^{\lambda}$, the lines of curvature maps into the lines of curvature and the asymptotic lines maps into the asymptotic lines. 
Since this deformation preserves the height differential 
\[
   {\phi}_{3} \, d\zeta =G(\zeta) \, \Psi(\zeta) d\zeta = \left(   \lambda G(\zeta)  \right) \, \left(  \frac{1}{\lambda} \Psi(\zeta) d\zeta  \right) = {\phi}^{\lambda}_{3} \, d\zeta,
\]
we immediately find that the third component of the conformal harmonic map is also preserved (up to vertical translations): 
\[
{\mathbf{x}}_{3}(\zeta)  = {\mathbf{x}}^{\lambda}_{3}(\zeta). 
\] 
Another important and useful property of the L\'{o}pez-Ros deformation is that if a component of a horizontal level set $\Sigma \cap \{ {\mathbf{x}}_{3} = 
\textrm{constant}\}$ is convex, then the same property holds for the related component at the corresponding height on ${\Sigma}^{\lambda}$. 
The L\'{o}pez-Ros deformation admits number of interesting applications \cite{BB2011, LR91, MP2012, PR1993, PR2002, Ros1996}.
\end{example}

 \section{Gauss map of degenerate minimal surfaces in ${\mathbb{R}}^{4}$}

Given a conformal harmonic immersion ${\mathbf{X}}=\left( x_0, x_1, x_2, x_3  \right) :\Sigma \rightarrow {\mathbb{R}}^{4}$ with 
the local coordinate $\zeta$,  we associate the complex curve $\phi \left(\zeta\right) : \Sigma \to {\mathbb{C}}^{4}$ defined by 
\begin{equation}  \label{to complex curve}
\phi = \left( {\phi}_{0}, {\phi}_{1},  {\phi}_{2}, {\phi}_{3} \right) := 2 \frac{d  {\mathbf{X}}}{d\zeta}.  
\end{equation}
We find that the harmonicity of $  {\mathbf{X}}$ guarantees that the curve $\phi$ is holomorphic and that the conformality of the immersion $ {\mathbf{X}}$ implies that $\phi$ lies on the complex null cone 
\begin{equation}  \label{complex null cone}
  {\mathcal{Q}}_{2}= \left\{  \left( z_0 , z_1, z_2, z_3 \right) \in {\mathbb{C}}^{4} \; \vert \; {z_0}^{2} + {z_1}^{2} +{z_2}^{2} +{z_3}^{2} = 0 \right\}.
\end{equation}
Such immersion ${\mathbf{X}}(\zeta)$ can be recovered, up to translations, by the integration  
\begin{equation}  \label{from complex curve}
  {\mathbf{X}} = \left(  \textrm{Re} \int {\omega}_{0}, \,  \textrm{Re} \int {\omega}_{1}, \, \textrm{Re} \int {\omega}_{2}, \, \textrm{Re} \int {\omega}_{3}\right), 
\end{equation}
where we introduce  holomorphic one forms $\left( {\omega}_{0}, {\omega}_{1},  {\omega}_{2}, {\omega}_{3} \right) = \left( {\phi}_{0} \, d\zeta, {\phi}_{1} \, d\zeta,  {\phi}_{2} \, d\zeta, {\phi}_{3} d\zeta \right)$. The induced metric on the surface $\Sigma$ reads 
$g_{\Sigma}=\frac{1}{2} \left( {\vert  {\phi}_{0}  \vert}^{2} +  {\vert  {\phi}_{1}  \vert}^{2} + {\vert  {\phi}_{2}  \vert}^{2} + {\vert  {\phi}_{3}  \vert}^{2} \right)
 \, {\vert d \zeta \vert}^{2}$.

We introduce the Gauss map of minimal surfaces in ${\mathbb{R}}^{4}$. Inside the complex projective space  ${\mathbb{C}}{\mathbb{P}}^{3}$, we take the complex null cone 
\[
  {\mathcal{Q}}_{2}:=\{ z=[{z}_{0}: {z}_{1}: {z}_{2}: {z}_{3} ] \in  {\mathbb{C}}{\mathbb{P}}^{3}  \; \vert \;   
  {{z}_{0}}^{2} + {{z}_{1}}^{2}  + {{z}_{2}}^{2}  +  {{z}_{3}}^{2} =0 \}.
\]

\begin{definition}[\textbf{Gauss map of minimal surfaces in ${\mathbb{R}}^{4}$ \cite[Chapter 3, p.35]{HO80}}]
Let $\Sigma$ be a minimal surface in ${\mathbb{R}}^{4}$. Consider a conformal harmonic immersion $  {\mathbf{X}} :\Sigma \rightarrow {\mathbb{R}}^{4}$ 
with the local coordinate $\zeta$. The Gauss map of $\Sigma$ is the map
$\mathcal{G}:\Sigma \rightarrow {\mathcal{Q}}_{2} \subset {\mathbb{C}}{\mathbb{P}}^{3}$ defined by
\[
   \mathcal{G}(\zeta):=\left[ \;   { \frac{\partial   {\mathbf{X}} }{\partial {\zeta}} } \; \right] 
   = [\; {\phi}_{0}:  {\phi}_{1}:  {\phi}_{2}:  {\phi}_{3} \; ] \in  {\mathbb{C}}{\mathbb{P}}^{3}.
\]
\end{definition}
 
 Following \cite[Chapter 4]{HO80}, we shall review the notion of degeneracy of Gauss map of minimal surfaces in ${\mathbb{R}}^{4}$. Let $\Sigma$ be a minimal surface,  which \textit{lies fully} in ${\mathbb{R}}^{4}$, in the sense that it does not lie in any proper affine subpace of ${\mathbb{R}}^{4}$.
 The surface $\Sigma$ is \textit{degenerate} when the image of its Gauss map lies in a hyperplane of ${\mathbb{CP}}^{3}$. Otherwise, 
it is  \textit{non-degenerate}. A degenerate minimal surface is  $k$-\textit{degenerate}, if $k \in \{1, 2, 3 \}$ is the largest  integer such that 
the Gauss map image lies in a projective subspace  of codimension $k$ in ${\mathbb{CP}}^{3}$. We recall the fundamental theorem \cite[Proposition 4.6]{HO80} that the surface $\Sigma$ is $2$-degenerate if and only if there exists an orthogonal complex structure on ${\mathbb{R}}^{4}$
such that it becomes a complex analytic curve lying fully in ${\mathbb{C}}^{2}$. 
 
Bernstein's beautiful theorem says that the only entire minimal graphs in ${\mathbb{R}}^{3}$ are planes.  
More generally, Osserman solved the codimension two generalization of Bernstein type problem,  and showed that 
examples of $1$-degenrate and $2$-generate minimal surfaces  in ${\mathbb{R}}^{4}$ naturally appear in the 
classification of entire minimal minimal graphs.
 
\begin{example}[\textbf{Osserman's entire, non-planar, minimal graph in ${\mathbb{R}}^{4}$ \cite[Chapter 5]{Oss86}}] 
We prepare a complex constant $\mu = a- ib$ with $a \in \mathbb{R}$ and $b>0$. For any entire holomorphic 
function $\mathbf{F}: \mathbb{C} \to \mathbb{C}$, we define the minimal surface ${\Sigma}$ with the patch
    \begin{center}
$ {\mathbf{X}} \left( \zeta \right)   
 = \left( \textrm{Re} \int   {\widehat{\,\phi\,}}_{0}\left( \zeta \right) d\zeta, \,    \textrm{Re} \int   {\widehat{\,\phi\,}}_{1}\left( \zeta \right) d\zeta, \, \textrm{Re} \int   {\widehat{\,\phi\,}}_{2}\left( \zeta \right) d\zeta, \, \textrm{Re} \int   {\widehat{\,\phi\,}}_{3}\left( \zeta \right) d\zeta \right)$,
  \end{center}
  where the holomorphic curve  $\left( {\widehat{\,\phi\,}}_{0} , \, {\widehat{\,\phi\,}}_{1} = \mu {\widehat{\,\phi\,}}_{0}, \,  {\widehat{\,\phi\,}}_{2}, {\widehat{\,\phi\,}}_{3}\right)$ reads
    \begin{center}
 $\left( 1, \,  {\mu}, \,  \frac{1}{2} \left(  e^{\mathbf{F}\left( \zeta \right)}  - \left( 1+ {\mu}^{2}  \right)  e^{-\mathbf{F}\left( \zeta \right)}     \right) , \,  
 \frac{i}{2} \left(  e^{\mathbf{F}\left( \zeta \right)}  +  \left( 1+ {\mu}^{2}  \right)  e^{-\mathbf{F}\left( \zeta \right)}     \right)  \right)$.
  \end{center}
 One can check that $\Sigma$ becomes the entire graph $\left(x_{1}, x_{2}, \mathbf{A}(x_{1}, x_{2}), \mathbf{B}(x_{1},x_{2}) \right)$ defined on the whole 
 $x_{1} x_{2}$-plane. 
 Osserman proved that any entire, non-planar, minimal graphs in ${\mathbb{R}}^{4}$ should admit the above representation with the entire holomorphic function $\mathbf{F}$. When $\mu \in \{i, -i \}$, the minimal surface $\Sigma$ becomes the complex analytic curve. 
\end{example}

\begin{remark}
As known in \cite[Theorem 4.7]{HO80}, degenerate minimal surfaces  in ${\mathbb{R}}^{4}$ admit a general representation formula, which is analogous to the classical Enneper-Weierstrass representation formula for minimal surfaces  in ${\mathbb{R}}^{3}$.
\end{remark}

 \section{Complex parabolic rotations of holomorphic null curves in ${\mathbb{C}}^{4}$}
 
Given a holomorphic null curve $\phi$ in  ${\mathbb{C}}^{4}$, for any linear mapping $\mathcal{M} \in \mathbf{O}\left(4, \mathbb{C} \right)$, we can associate new holomorphic null curve  $\widetilde{\phi} = M \phi$. The purpose of this section is to construct the so-called 
 parabolic rotations of holomorphic null curves in  ${\mathbb{C}}^{4}$ to construct explicit deformations of minimal surfaces 
 in ${\mathbb{R}}^{3}$ to degenerate minimal surfaces in ${\mathbb{R}}^{4}$.

\begin{lemma}[\textbf{Complex parabolic rotations of holomorphic null curves in ${\mathbb{C}}^{4}$}] \label{deform 1} 
Given a non-constant holomorphic curve $\phi \left(\zeta\right): \Sigma \to {\mathbb{C}}^{4}$ and a constant $c \in \mathbb{C}$, we associate the holomorphic curve $\widehat{\phi}: \Sigma \to {\mathbb{C}}^{4}$ by the linear transformation 
\begin{equation}  \label{parabolic linear transformation on complex null cone}
 \begin{bmatrix} 
 {\phi}_{0} \\ {\phi}_{1} \\  {\phi}_{2} \\ {\phi}_{3}
 \end{bmatrix} 
 \mapsto 
  \begin{bmatrix} 
 {\widehat{\phi}}_{0} \\  {\widehat{\phi}}_{1} \\   {\widehat{\phi}}_{2} \\  {\widehat{\phi}}_{3}
 \end{bmatrix} 
 =
 \begin{bmatrix} 
 1 &  - c  & -c \, i & 0 \\ 
  c  & 1-\frac{c^2}{2}  & - \frac{c^2}{2} i & 0 \\ 
 c \, i  & - \frac{c^2}{2} i & 1+ \frac{c^2}{2} & 0 \\ 
 0 & 0 & 0 & 1 
 \end{bmatrix} 
  \begin{bmatrix} 
 {\phi}_{0} \\ {\phi}_{1} \\  {\phi}_{2} \\ {\phi}_{3}
 \end{bmatrix} 
\end{equation} 
 
If the curve $\phi$ lies on the null cone  ${\mathcal{Q}}_{2} = \left\{  \left( z_0 , z_1, z_2, z_3 \right) \in {\mathbb{C}}^{4} \; \vert \; {z_0}^{2} + {z_1}^{2} +{z_2}^{2} +{z_3}^{2} = 0 \right\}$, then the curve $\widehat{\phi}$ also lies on ${\mathcal{Q}}_{2}$. 
\end{lemma}

\begin{proof} It is straightforward to check the algebraic identity
\begin{equation}   \label{parabolic rotation}
  { {\widehat{\phi}}_{0} \, }^{2} +   { {\widehat{\phi}}_{1} \, }^{2} +  { {\widehat{\phi}}_{2} \, }^{2}  =  { {\phi}_{0} }^{2} +  { {\phi}_{1} }^{2} + { {\phi}_{2} }^{2}.  \end{equation}
Since  ${\widehat{\phi}}_{3} =  {\phi}_{3}$, this implies 
\[
{ {\widehat{\phi}}_{0} \, }^{2} +   { {\widehat{\phi}}_{1} \, }^{2} +  { {\widehat{\phi}}_{2} \, }^{2} +  { {\widehat{\phi}}_{3} \, }^{2}=  { {\phi}_{0} }^{2} +  { {\phi}_{1} }^{2} + { {\phi}_{2} }^{2} + { {\phi}_{3} }^{2}.
\] 
Since the curve $\phi$ lies on the null cone ${\mathcal{Q}}_{2}$, the curve $\widehat{\phi}$ also lies on ${\mathcal{Q}}_{2}$. 
\end{proof}

The identity (\ref{parabolic rotation}) in the proof of Lemma \ref{deform 1} is the key idea of our deformation (\ref{parabolic linear transformation on complex null cone}) in Lemma \ref{deform 1}. We shall illustrate ideas behind the identity (\ref{parabolic rotation}) and Lemma \ref{deform 1}. 
 
\begin{remark}[\textbf{Real light cone $x^2 +y^2 -z^2=0$ in ${\mathbb{L}}^{3}$, complex cone ${z_{0}}^{2} +{z_{1}}^{2}+{z_{2}}^{2}=0$ in 
${\mathbb{C}}^{3}$,  parabolic rotations, and Wick rotation}]  \label{wick} We explain that the algebraic identity (\ref{parabolic rotation}) in the proof of Lemma \ref{deform 1} can be obtained by the complexification of parabolic rotational isometries  in ${\mathbb{L}}^{3}$ via Wick rotation. Let  ${\mathbb{L}}^{3}$ denote the Lorentz-Minkowski $(2+1)$-space, which is the real vector space ${\mathbb{R}}^{3}$ endowed with the Lorentzian metric $dx^2 +dy^2 -dz^2$. The light cone  sitting in  ${\mathbb{L}}^{3}$ given by the quadratic variety 
\[
  \left\{  \left( x, y, z \right) \in {\mathbb{R}}^{3}  \; \vert \; x^2 +y^2 -z^2=0 \right\}
\]
is invariant under parabolic rotations $ {\mathcal{L}}_{t \in \mathbb{R}}$ with respect to the null line spanned by light-like vector $(1, 0, 1)$. The 
isometry ${\mathcal{L}}_{t}$ is given by  (for instance, see \cite[Section 2]{MP03})
\begin{equation}  \label{parabolic rotation 2}
 \begin{bmatrix} 
x  \\  y \\ z
 \end{bmatrix} 
 \mapsto 
  \begin{bmatrix} 
 \widehat{x}   \\   \widehat{y} \\  \widehat{z}
 \end{bmatrix} 
 = {\mathcal{L}}_{t} \begin{pmatrix} 
 \widehat{x}   \\   \widehat{y} \\  \widehat{z}
 \end{pmatrix} 
 =
 \begin{bmatrix} 
1-\frac{t^2}{2}   &  t &  \frac{t^2}{2}  \\ 
  -t  & 1   & t \\ 
- \frac{t^2}{2} & t &  1+\frac{t^2}{2} 
 \end{bmatrix} 
  \begin{bmatrix} 
x  \\  y \\ z
 \end{bmatrix},
\end{equation}
for each $t \in \mathbb{R}$. We have the real quadratic form identity
\begin{equation}  \label{minus cone}
x^2 +y^2 -z^2 =  {\widehat{x}\,}^2 + {\widehat{y}\,}^2 -{\widehat{z}\,}^2, \; \text{where} \begin{bmatrix} 
 \widehat{x}   \\   \widehat{y} \\  \widehat{z}
 \end{bmatrix} 
 = {\mathcal{L}}_{t} \begin{pmatrix} 
 \widehat{x}   \\   \widehat{y} \\  \widehat{z}
 \end{pmatrix}. 
\end{equation}
So far, we viewed $x$, $y$, $z$,  $\widehat{x}$,  $\widehat{y}$, $\widehat{z}$ as real variables. 
However, clearly, the algebraic identity (\ref{minus cone}) also holds when we treat them as complex variables. We perform
the so-called Wick rotation. Replace 
$z$ by $-i z$, and $\widehat{z}$ by $- i \widehat{z}$ in (\ref{parabolic rotation}) to have the transformation
\[
 \begin{bmatrix} 
x  \\  y \\ -i z
 \end{bmatrix} 
 \mapsto 
  \begin{bmatrix} 
 \widehat{x}   \\   \widehat{y} \\  -i \widehat{z}
 \end{bmatrix} 
 =
 \begin{bmatrix} 
1-\frac{t^2}{2}   &  t &  \frac{t^2}{2}  \\ 
  -t  & 1   & t \\ 
- \frac{t^2}{2} & t &  1+\frac{t^2}{2} 
 \end{bmatrix} 
  \begin{bmatrix} 
x  \\  y \\ -i z
 \end{bmatrix},
\]
which can be rewritten as 
\[
 \begin{bmatrix} 
x  \\  y \\   z
 \end{bmatrix} 
 \mapsto 
  \begin{bmatrix} 
 \widehat{x}   \\   \widehat{y} \\    \widehat{z}
 \end{bmatrix} 
 = {\mathcal{R}}_{t} \begin{pmatrix} 
x  \\ y \\  z
 \end{pmatrix} 
 =
 \begin{bmatrix} 
1-\frac{t^2}{2} &  t  & - \frac{t^2}{2} i  \\ 
 - t  &  1 & - t i   \\ 
- \frac{t^2}{2} i & t i & 1+ \frac{t^2}{2} 
  \end{bmatrix} 
  \begin{bmatrix} 
x  \\  y \\  z
 \end{bmatrix},
\]
Now, the real variables identity (\ref{minus cone}) induces the complex quadratic form identity 
\begin{equation}  \label{plus cone}
x^2 +y^2 +z^2 =  {\widehat{x}\,}^2 + {\widehat{y}\,}^2 +{\widehat{z}\,}^2,  \; \text{where} \begin{bmatrix} 
 \widehat{x}   \\   \widehat{y} \\  \widehat{z}
 \end{bmatrix} 
 = {\mathcal{R}}_{t} \begin{pmatrix} 
 \widehat{x}   \\   \widehat{y} \\  \widehat{z}
 \end{pmatrix}. 
\end{equation}
Finally, taking $(x, y, z)=\left(z_1, z_0, z_2 \right)$, $(\widehat{x}, \widehat{y}, \widehat{z})=\left({\widehat{z}}_1, {\widehat{z}}_0, {\widehat{z}}_2 \right)$, and $t=c$, we have
\begin{equation} 
{z_0}^2 +{z_1}^2 +{z_2}^2 =  { {\widehat{z}}_{0}\,}^2 + { {\widehat{z}}_{1}\,}^2 +{ {\widehat{z}}_{2}\,}^2,  \; \text{where} \;
 \begin{bmatrix} 
 {\widehat{z}}_{0} \\  {\widehat{z}}_{1} \\   {\widehat{z}}_{2}  
  \end{bmatrix} 
 =
 \begin{bmatrix} 
 1 &  - c  & -c \, i   \\ 
  c  & 1-\frac{c^2}{2}  & - \frac{c^2}{2} i   \\ 
 c \, i  & - \frac{c^2}{2} i & 1+ \frac{c^2}{2}  
  \end{bmatrix} 
  \begin{bmatrix} 
 {z}_{0} \\ {z}_{1} \\  {z}_{2}  
 \end{bmatrix}, 
\end{equation} 
which gives the algebraic identity (\ref{parabolic rotation}). We see that $\left(z_1, z_0, z_2 \right) \mapsto \left({\widehat{z}}_1, {\widehat{z}}_0, {\widehat{z}}_2 \right)$ becomes the well-defined linear transformation from the complex null cone to itself. See also \cite[Section 3. Complex rotations and minimal surfaces]{MB2000}.
\end{remark}

\begin{remark}[\textbf{Parabolic rotations of null curves in ${\mathbb{C}}^{4}$ in terms of Segre coordinates}]
We begin with the algebraic identity 
\begin{equation} \label{segre product}
{z_0}^{2} + {z_1}^{2} +{z_2}^{2} +{z_3}^{2} =  \left( z_{1} + i z_{2} \right) \left(  z_{1} - i z_{2} \right) - \left( z_{3} + i z_{0} \right) \left( -z_{3} + i z_{0} \right).
\end{equation}  
Using the Segre transformation $\left( t_{1},  t_{2}, t_{3}, t_{0} \right) = \left(  z_{1} + i z_{2},  z_{1} - i z_{2}, z_{3} + i z_{0},  -z_{3} + i z_{0}\right)$, we are able to identify the  null cone ${z_0}^{2} + {z_1}^{2} +{z_2}^{2} +{z_3}^{2} =0$ as the determinant variety  
\begin{equation}  \label{complex null cone}
\mathcal{S} = \left\{ \, \begin{bmatrix} 
 t_{1} & t_{0}    \\  
 t_{3} & t_{2} 
 \end{bmatrix}  \in M_{2} \left( \mathbb{C} \right) \; \vert \; \det  \begin{bmatrix} 
 t_{1} & t_{0}    \\  
 t_{3} & t_{2} 
 \end{bmatrix}  = 0 \, \right\}.
\end{equation}
Given a pair $\left(L, R\right)$ of complex constants, we find the implication 
\begin{equation}  \label{generalized deformation in serge}
\begin{bmatrix} 
 t_{1} & t_{0}    \\  
 t_{3} & t_{2} 
 \end{bmatrix} \in \mathcal{S}
\; \Rightarrow \; 
 \begin{bmatrix} 
 \, {\widehat{t}}_{1} &  {\widehat{t}}_{0}   \\  
 \, {\widehat{t}}_{3} & {\widehat{t}}_{2}
 \end{bmatrix} 
 := \begin{bmatrix} 
 1 & 0    \\  
  L &1
 \end{bmatrix} 
 \begin{bmatrix} 
 t_{1} & t_{0}    \\  
 t_{3} & t_{2} 
 \end{bmatrix} 
 \begin{bmatrix} 
1 & R   \\  
0 & 1 
 \end{bmatrix} 
 \in \mathcal{S}.
 \end{equation}
Writing $\left(  {\widehat{t}}_{1},   {\widehat{t}}_{2},  {\widehat{t}}_{3},  {\widehat{t}}_{0} \right) = 
\left(   {\widehat{z}}_{1} + i  {\widehat{z}}_{2},   {\widehat{z}}_{1} - i  {\widehat{z}}_{2}, {\widehat{z}}_{3}
 + i  {\widehat{z}}_{0},  - {\widehat{z}}_{3} + i  {\widehat{z}}_{0}\right)$, we see that this implication induces 
 the linear map from the quadratic cone ${z_0}^{2} + {z_1}^{2} +{z_2}^{2} +{z_3}^{2} =0$ to itself:
 \begin{equation}  \label{LR}
 \begin{bmatrix} 
 {z}_{0} \\ {z}_{1} \\  {z}_{2} \\ {z}_{3}
 \end{bmatrix} 
 \mapsto 
  \begin{bmatrix} 
 {\widehat{z}}_{0} \\  {\widehat{z}}_{1} \\   {\widehat{z}}_{2} \\  {\widehat{z}}_{3}
 \end{bmatrix} 
 =
 \begin{bmatrix} 
 1 &  - \frac{i}{2} \left( L+R \right)  & \frac{1}{2}  \left( L+R \right) & 0 \\ 
\frac{i}{2} \left( L+R \right) & 1 + \frac{LR}{2}  & \frac{i}{2} LR & - \frac{1}{2}   \left( L -R \right)  \\ 
-  \frac{1}{2}  \left( L+R \right)   &   \frac{i}{2} LR & 1 -  \frac{LR}{2} &  - \frac{i}{2}   \left( L -R \right) \\ 
 0 & \frac{1}{2}   \left( L -R \right) &  \frac{i}{2}   \left( L -R \right) & 1 
 \end{bmatrix} 
 \begin{bmatrix} 
 {z}_{0} \\ {z}_{1} \\  {z}_{2} \\ {z}_{3}
 \end{bmatrix} 
\end{equation} 
In particular, taking $\left(L, R\right)=\left(- c i, -ci\right)$, we obtain the linear transformation 
  \begin{equation}  \label{LRc}
 \begin{bmatrix} 
 {z}_{0} \\ {z}_{1} \\  {z}_{2} \\ {z}_{3}
 \end{bmatrix} 
 \mapsto 
  \begin{bmatrix} 
 {\widehat{z}}_{0} \\  {\widehat{z}}_{1} \\   {\widehat{z}}_{2} \\  {\widehat{z}}_{3}
 \end{bmatrix} 
 =
 \begin{bmatrix} 
  1 &  - c  & -c \, i & 0 \\ 
  c  & 1-\frac{c^2}{2}  & - \frac{c^2}{2} i & 0 \\ 
 c \, i  & - \frac{c^2}{2} i & 1+ \frac{c^2}{2} & 0 \\ 
 0 & 0 & 0 & 1 
  \end{bmatrix} 
 \begin{bmatrix} 
 {z}_{0} \\ {z}_{1} \\  {z}_{2} \\ {z}_{3}
 \end{bmatrix}, 
\end{equation} 
 which recovers the deformation (\ref{parabolic linear transformation on complex null cone}) in Lemma \ref{deform 1}. 
\end{remark}
  
 \section{From minimal surfaces in ${\mathbb{R}}^{3}$ to degenerate minimal surfaces in ${\mathbb{R}}^{4}$}
    
\begin{theorem}[\textbf{Deformations of non-planar minimal surfaces in ${\mathbb{R}}^{3}$ to degenerate minimal surfaces in ${\mathbb{R}}^{4}$}] \label{thm:main}
 Let $\Sigma$ be a simply connected non-planar minimal surface in ${\mathbb{R}}^{3}$, up to translations, parametrized by the 
 conformal harmonic immersion
  \begin{equation}
  {\mathbf{X}}\left( \zeta \right)   
 = \left(    \textrm{Re} \int  {\phi}_{1}\left( \zeta \right)    d\zeta, \, \textrm{Re} \int  {\phi}_{2}\left( \zeta \right)    d\zeta, \, \textrm{Re} \int  {\phi}_{3}\left( \zeta \right)    d\zeta \right), 
  \end{equation}
  where the holomorphic null curve $  {\phi} = \left(   {\phi}_{1}, \,   {\phi}_{2}, \,   {\phi}_{3}\right)$ admits the Weierstrass data $(G(\zeta), \Psi(\zeta) d\zeta):$ 
      \begin{equation} \label{R3 w data}
 {\phi} = 
   \left(  \frac{1}{2}  \left(1 -G^2 \right)  \Psi, \,  \frac{i}{2} \left(1 +G^2 \right)  \Psi, \,   G  \Psi  \right),
  \end{equation}
  Given a constant $c \in \mathbb{C}$, there exists a minimal surface ${\Sigma}^{c}$ in ${\mathbb{R}}^{4}$, up to translations, parametrized by the conformal harmonic immersion
    \begin{equation}
  {\mathbf{X}}^{c} \left( \zeta \right)   
 = \left( \textrm{Re} \int   {\widehat{\,\phi\,}}_{0}\left( \zeta \right) d\zeta, \,    \textrm{Re} \int   {\widehat{\,\phi\,}}_{1}\left( \zeta \right) d\zeta, \, \textrm{Re} \int   {\widehat{\,\phi\,}}_{2}\left( \zeta \right) d\zeta, \, \textrm{Re} \int   {\widehat{\,\phi\,}}_{3}\left( \zeta \right) d\zeta \right),
  \end{equation}
  where the holomorphic curve  $    {\widehat{\,\phi \,}} = \left( {\widehat{\,\phi\,}}_{0}, \, {\widehat{\,\phi\,}}_{1}, \,  {\widehat{\,\phi\,}}_{2}, {\widehat{\,\phi\,}}_{3}\right)$ is 
  determined by 
  \begin{equation}  \label{deformation w data}
   {\widehat{\,\phi \,}}
= \left(  c \, G^2  \Psi, \,  \frac{1}{2}  \left(1 + \left(c^2 -1\right)G^2 \right)  \Psi  , \,  \frac{i}{2} \left(1 +\left(c^2 +1\right)G^2 \right)  \Psi, \,   G  \Psi    \right)
 \end{equation}
 The induced metric on ${\Sigma}^{c}$ by the patch ${\mathbf{X}}^{c}$ reads 
  \[
  g_{{}_{{\Sigma}^{c}}}=
  \frac{1}{4}\left( \,  {\vert \Psi \vert}^2 \, {\vert   1 + c^2 G^2   \vert}^{2}  \left( 1 + \frac{  {\vert G \vert}^{2} }{ { \vert 1 +i c G \vert  }^{2}} \right)  \left( 1 + \frac{  {\vert G \vert}^{2} }{ { \vert 1 -i c G \vert  }^{2}} \right)   \, \right)  \,  {\vert d\zeta \vert}^{2}.
  \]
  Due to the identity $ {\widehat{\,\phi\,}}_{0} + c  {\widehat{\,\phi\,}}_{1} +  i c   {\widehat{\,\phi\,}}_{2}=0$, the minimal 
  surface ${\Sigma}^{c}$ in ${\mathbb{R}}^{4}$ is degenerate. 
\end{theorem}

\begin{proof} 
To prove that ${\Sigma}^{c}$ is a minimal surface in ${\mathbb{R}}^{4}$, we show that the holomorphic curve 
\[
 \left(  c \, G^2  \Psi, \,  \frac{1}{2}  \left(1 + \left(c^2 -1\right)G^2 \right)  \Psi  , \,  \frac{i}{2} \left(1 +\left(c^2 +1\right)G^2 \right)  \Psi, \,   G  \Psi    \right)
 \]
 is null. Take ${\phi}_{0} =0$.  Regard  ${\Sigma}$ in ${\mathbb{R}}^{3}$ 
induced by the holomorphic null curve $\left({\phi}_{1}, {\phi}_{2}, {\phi}_{3} \right)$ in ${\mathbb{C}}^{3}$ as a minimal surface in ${\mathbb{R}}^{4}$ induced by the holomorphic null curve in ${\mathbb{C}}^{4}:$  
\[
\left({\phi}_{0} , {\phi}_{1}, {\phi}_{2}, {\phi}_{3} \right) = \left(0 , \frac{1}{2}  \left(1 -G^2 \right)  \Psi, \,  \frac{i}{2} \left(1 +G^2 \right)  \Psi, \, G  \Psi \right).
\]
Then, by Lemma \ref{deform 1}, we find that the holomorphic curve  in ${\mathbb{C}}^{4}:$ 
\[
 \begin{bmatrix} 
 1 &  - c  & -c \, i & 0 \\ 
  c  & 1-\frac{c^2}{2}  & - \frac{c^2}{2} i & 0 \\ 
 c \, i  & - \frac{c^2}{2} i & 1+ \frac{c^2}{2} & 0 \\ 
 0 & 0 & 0 & 1 
 \end{bmatrix} 
  \begin{bmatrix} 
 {\phi}_{0} \\ {\phi}_{1} \\  {\phi}_{2} \\ {\phi}_{3}
 \end{bmatrix} 
 =
   \begin{bmatrix} 
     c \, G^2  \Psi      \\    \frac{1}{2}  \left(1 + \left(c^2 -1\right)G^2 \right)  \Psi      \\  
    \frac{i}{2} \left(1 +\left(c^2 +1\right)G^2 \right)  \Psi      \\    G  \Psi  
 \end{bmatrix} 
 =  \begin{bmatrix} 
 {\widehat{\phi}}_{0} \\  {\widehat{\phi}}_{1} \\   {\widehat{\phi}}_{2} \\  {\widehat{\phi}}_{3}
 \end{bmatrix} 
 \]
should be also null. The conformal factor of the conformal metric on  ${\Sigma}^{c}$ is equal to
\[
 \frac{1}{2} \left(  \sum_{k=0}^{3} {\vert  {\widehat{\,\phi\,}}_{k}   \vert}^2  \right) =   \frac{1}{4}\left(  {\vert \Psi \vert}^2 \, {\vert   1 + c^2 G^2   \vert}^{2}  \left( 1 + \frac{  {\vert G \vert}^{2} }{ { \vert 1 +i c G \vert  }^{2}} \right)  \left( 1 + \frac{  {\vert G \vert}^{2} }{ { \vert 1 -i c G \vert  }^{2}} \right)  \right).
\]
The definition of ${\widehat{\,\phi \,}}$ gives the following equality, which implies the degeneracy of ${\Sigma}^{c}$: 
\[
   {\widehat{\,\phi\,}}_{0} + c  {\widehat{\,\phi\,}}_{1} +  i c   {\widehat{\,\phi\,}}_{2} = {\phi}_{0} =0. 
\]
\end{proof}

 \begin{remark}  \label{lawson def} 
 We  used the parabolic rotations of holomorphic null curves in ${\mathbb{C}}^{4}$ in Lemma \ref{deform 1} to construct 
 deformations  (Theorem \ref{thm:main}) of minimal surfaces in ${\mathbb{R}}^{3}$ to a family of minimal surfaces in  ${\mathbb{R}}^{4}$. 
  These deformations of simply connected minimal surfaces in  ${\mathbb{R}}^{3}$ to minimal surfaces in ${\mathbb{R}}^{3}$ 
or ${\mathbb{R}}^{4}$ are non-isometric deformations, in general.  Indeed,  H. Lawson \cite[Theorem 1]{Law71} used E. Calabi's Theorem \cite{Cal53} to determine when minimal surfaces 
in ${\mathbb{R}}^{n \geq 3}$ are isometric to minimal surfaces in ${\mathbb{R}}^{3}$. See also \cite[Theorem 1.2]{MM2015}.  
 \end{remark}
  
 \begin{remark}
We point out that the holomorphic null curves in ${\mathbb{C}}^{4}$ also naturally appears in the theory of superconformal surfaces in ${\mathbb{R}}^{4}$. For instance, see \cite{DT09, Mor09}. 
 \end{remark}
 
Taking the constant $c = \tan \theta \in \mathbb{R}$ in Theorem \ref{thm:main} and rotating coordinate system in the ambient space
 ${\mathbb{R}}^{4}$, we have the following deformation:
   
 \begin{corollary}[\textbf{Degenerate minimal surfaces in ${\mathbb{R}}^{4}$}]   \label{main:thm2}
 Let $\Sigma$ be a simply connected minimal surface in ${\mathbb{R}}^{3}$, up to translations, parametrized by the conformal harmonic immersion
  \begin{equation}
  {\mathbf{X}}\left( \zeta \right)   
 = \left(    \textrm{Re} \int  {\phi}_{1}\left( \zeta \right)    d\zeta, \, \textrm{Re} \int  {\phi}_{2}\left( \zeta \right)    d\zeta, \, \textrm{Re} \int  {\phi}_{3}\left( \zeta \right)    d\zeta \right), 
  \end{equation}
  where the holomorphic null curve $  {\phi} = \left(   {\phi}_{1}, \,   {\phi}_{2}, \,   {\phi}_{3}\right)$ admits the Weierstrass data $(G(\zeta), \Psi(\zeta) d\zeta):$ 
      \begin{equation} \label{R3 w data}
 {\phi} = 
   \left(  \frac{1}{2}  \left(1 -G^2 \right)  \Psi, \,  \frac{i}{2} \left(1 +G^2 \right)  \Psi, \,   G  \Psi  \right),
  \end{equation}
 For each angle constant $\theta \in \left( -\frac{\pi}{2}, \frac{\pi}{2} \right)$, there exists a degenerate minimal surface ${\Sigma}^{\tan \theta}$ in ${\mathbb{R}}^{4}$, up to translations, parametrized by the conformal harmonic immersion
    \begin{equation}
  {\mathbf{X}}^{\tan \theta} \left( \zeta \right)   
 = \left( \textrm{Re} \int   {\widetilde{\,\phi\,}}_{0}\left( \zeta \right) d\zeta, \,    \textrm{Re} \int   {\widetilde{\,\phi\,}}_{1}\left( \zeta \right) d\zeta, \, \textrm{Re} \int   {\widetilde{\,\phi\,}}_{2}\left( \zeta \right) d\zeta, \, \textrm{Re} \int   {\widetilde{\,\phi\,}}_{3}\left( \zeta \right) d\zeta \right),
  \end{equation}
  where the holomorphic curve   $\left( {\widetilde{\,\phi\,}}_{0}= -i \left( \sin \theta \right) {\widehat{\,\phi\,}}_{2}, \, {\widetilde{\,\phi\,}}_{1}, \,  {\widetilde{\,\phi\,}}_{2}, {\widehat{\,\phi\,}}_{3}\right)$ is 
  determined by 
  \begin{equation}  
   {\widetilde{\,\phi \,}}
= \left(   \frac{\sin \theta}{2} \left(1 + \frac{ G^2}{{\cos}^{2} \theta} \right)   \Psi , \,  \frac{\cos \theta}{2}  \left(1 - \frac{ G^2}{{\cos}^{2} \theta} \right)  \Psi  , \,  \frac{i}{2} \left(1 + \frac{ G^2}{{\cos}^{2} \theta} \right)   \Psi, \,   G  \Psi    \right). 
 \end{equation}
  \end{corollary} 

\begin{proof}
We take $c=\tan \theta$  in Theorem \ref{thm:main}. With respect to the standard frame ${\mathbf{e}}_{0}=(1,0,0,0)$, ${\mathbf{e}}_{1}=(0,1,0,0)$, ${\mathbf{e}}_{2}=(0,0,1,0)$, ${\mathbf{e}}_{3}=(0,0,0,1)$, the surface ${\Sigma}^{\tan \theta}$ admits the patch 
\[
 \left( \textrm{Re} \int   {\widehat{\phi}}_{0}\left( \zeta \right) d\zeta \right) {\mathbf{e}}_{0}
+ \left(  \textrm{Re} \int   {\widehat{\phi}}_{1}\left( \zeta \right) d\zeta \right) {\mathbf{e}}_{1} 
+ \left( \textrm{Re} \int   {\widehat{\phi}}_{2}\left( \zeta \right) d\zeta \right) {\mathbf{e}}_{2}
+\left( \textrm{Re} \int   {\widehat{\phi}}_{3}\left( \zeta \right) d\zeta \right) {\mathbf{e}}_{3},
\]
  where the holomorphic curve  $    {\widehat{\,\phi \,}} = \left( {\widehat{\,\phi\,}}_{0}, \, {\widehat{\,\phi\,}}_{1}, \,  {\widehat{\,\phi\,}}_{2}, {\widehat{\,\phi\,}}_{3}\right)$ reads
\[
   {\widehat{\,\phi \,}}
= \left(  \tan \theta \, G^2  \Psi, \,  \frac{1}{2}  \left(1 + \left(-1 + {\tan}^2 \theta \right)G^2 \right)  \Psi  , \,  \frac{i}{2} \left(1 +\left(1+ {\tan}^2 \theta  \right)G^2 \right)  \Psi, \,   G  \Psi    \right).
\]
With the new frame ${\mathbf{E}}_{0}=\cos \theta\, {\mathbf{e}}_{0} + \sin \theta\, {\mathbf{e}}_{1}$, ${\mathbf{E}}_{1}=-\sin \theta\, {\mathbf{e}}_{0} + \cos \theta\, {\mathbf{e}}_{1}$, ${\mathbf{E}}_{2}={\mathbf{e}}_{2}$, $ {\mathbf{E}}_{3}={\mathbf{e}}_{3}$, the minimal surface ${\Sigma}^{\tan \theta}$ admits 
the patch 
\[
 \left( \textrm{Re} \int   {\widetilde{\phi}}_{0}\left( \zeta \right) d\zeta \right) {\mathbf{E}}_{0}
+ \left(  \textrm{Re} \int   {\widetilde{\phi}}_{1}\left( \zeta \right) d\zeta \right) {\mathbf{E}}_{1} 
+ \left( \textrm{Re} \int   {\widetilde{\phi}}_{2}\left( \zeta \right) d\zeta \right) {\mathbf{E}}_{2}
+\left( \textrm{Re} \int   {\widetilde{\phi}}_{3}\left( \zeta \right) d\zeta \right) {\mathbf{E}}_{3}.
\]
\end{proof}

 \section{Minimal surfaces in ${\mathbb{R}}^{4}$ foliated by conic sections}
  
We present applications of our deformations of minimal surfaces in ${\mathbb{R}}^{3}$ to produce new minimal surfaces in ${\mathbb{R}}^{4}$. 
Applying our deformation to holomorphic null curves in ${\mathbb{C}}^{3} \subset {\mathbb{C}}^{4}$ induced by helicoids in ${\mathbb{R}}^{3}$, we discover minimal surfaces in ${\mathbb{R}}^{4}$ foliated by conic sections with eccentricity grater than $1$: hyperbolas or straight lines. Applying our deformation to   holomorphic null curves in ${\mathbb{C}}^{3}$ induced by catenoids in ${\mathbb{R}}^{3}$, we can rediscover the Hoffman-Osserman minimal surfaces in ${\mathbb{R}}^{4}$ foliated by conic sections with eccentricity smaller than $1$: ellipses or circles.

\begin{example}[\textbf{From helicoids in ${\mathbb{R}}^{3}$ to minimal surfaces in ${\mathbb{R}}^{4}$ foliated by conic sections with eccentricity grater than $1$: hyperbolas or straight lines}] \label{deformation of helicoids} Let's begin with the helicoid $\Sigma = \{ \left( -\sinh u \sin v, \sinh u \cos v, v \right) \in {\mathbb{R}}^{3}  \;
\vert \; u, v \in \mathbb{R} \}$ foliated by horizontal lines. The global conformal coordinate $\zeta=u+iv \in \mathbb{C}$ on $\Sigma$ induces the Weierstrass data 
\[
\left(G(\zeta), \Psi(\zeta)   d\zeta\right)=\left(  e^{\zeta}, -i e^{-\zeta}   d\zeta\right).
\]
The helicoid $\Sigma$ is, up to translations, given by the conformal harmonic immersion
\[
  {\mathbf{X}}\left( \zeta \right)   
 =  \textrm{Re}  \left(    \int  \frac{1}{2}  \left(1 -G^2 \right)  \Psi d\zeta, \,  \int \frac{i}{2} \left(1 +G^2 \right)  \Psi d\zeta, \,  \int  G  \Psi d\zeta\right).
\]
 Let $c = \alpha + i \beta \in \mathbb{R} + i  \mathbb{R}$ be the deformation constant. Applying Theorem \ref{thm:main} to the Weierstrass data $\left(G(\zeta), \Psi(\zeta)\right)=\left(  e^{\zeta}, -i e^{-\zeta} \right)$, we obtain the minimal surface ${\Sigma}^{c}$ in ${\mathbb{R}}^{4}$, up to translations, parametrized by the conformal harmonic map ${\mathbf{X}}^{c} \left( \zeta \right):$   
\[
  {\mathbf{X}}^{c}  
 =   \textrm{Re} \left( \int  c \, G^2  \Psi d\zeta, \,   \int \frac{1}{2}  \left(1 + \left(c^2 -1\right)G^2 \right)  \Psi  d\zeta, \,   \int 
 \frac{i}{2} \left(1 +\left(c^2 +1\right)G^2 \right)  \Psi d\zeta, \,  \int G  \Psi d\zeta \right).
\]
With the frame ${\mathbf{e}}_{0}=(1,0,0,0)$, ${\mathbf{e}}_{1}=(0,1,0,0)$, ${\mathbf{e}}_{2}=(0,0,1,0)$, ${\mathbf{e}}_{3}=(0,0,0,1)$, we write 
\[
  {\mathbf{X}}^{c}  =   {\mathbf{X}}_0 {\mathbf{e}}_0 +   {\mathbf{X}}_1 {\mathbf{e}}_1 +   {\mathbf{X}}_2 {\mathbf{e}}_2+   {\mathbf{X}}_3 {\mathbf{e}}_3
\]
and obtain 
\begin{eqnarray*}
  {\mathbf{X}}_{0}(u, v) &=&  e^u \left( \alpha \sin v + \beta \cos v \right), \\
  {\mathbf{X}}_{1}(u, v) &=&  e^u \left( \frac{  {\alpha}^{2} -{\beta}^{2} - 1  }{2} \sin v + \alpha \beta \cos v \right) + e^{-u} \left(  \frac{\sin v}{2} \right), \\
  {\mathbf{X}}_{2}(u, v) &=&  e^u \left( \frac{ -  \alpha \beta \sin v + {\alpha}^{2} -{\beta}^{2} +1  }{2} \cos v  \right) + e^{-u} \left(  \frac{-\cos v}{2} \right), \\
  {\mathbf{X}}_{3}(u, v) &=&  v. 
\end{eqnarray*}
We show that the minimal surface ${\Sigma}^{c}$ is foliated by hyperbolas or lines. 
We introduce new orthonormal frame 
 \begin{eqnarray*}
   {\mathbf{E}}_0 &=&   \frac{1}{\sqrt{ {\alpha}^{2}+{\beta}^{2}+1}} \left( e_{0} + \alpha e_{1} - \beta e_{2}   \right), \\
  {\mathbf{E}}_1 &=&      \frac{1}{\sqrt{{\alpha}^{2} +1}} \left( - \alpha e_{0} +   e_{1}   \right),  \\
   {\mathbf{E}}_2  &=&    \frac{1}{\sqrt{{\beta}^{2} +1}} \left( \, \beta e_{0} + e_{2}   \right), \\
   {\mathbf{E}}_3 &=&  e_{3},
  \end{eqnarray*}
and prepare  two auxiliary functions
\begin{eqnarray*}
 {\mathbf{Ch}}(u) &=& \frac{1}{2} \left( \sqrt{ {\alpha}^{2}+{\beta}^{2}+1} \, e^u + \frac{1}{ \sqrt{ {\alpha}^{2}+{\beta}^{2}+1} \, e^u} \right) 
 =\cosh \left( u + \ln \left(  \sqrt{ {\alpha}^{2}+{\beta}^{2}+1} \right)  \right), \\
 {\mathbf{Sh}}(u) &=& \frac{1}{2} \left( \sqrt{ {\alpha}^{2}+{\beta}^{2}+1} \, e^u - \frac{1}{ \sqrt{ {\alpha}^{2}+{\beta}^{2}+1} \, e^u} \right)
 =\sinh \left( u + \ln \left(  \sqrt{ {\alpha}^{2}+{\beta}^{2}+1} \right)  \right). \\
\end{eqnarray*} 
Each components of the patch ${\mathbf{X}}^{c}(u, v)  =   {\Xi}_0 {\mathbf{E}}_0 +   {\Xi}_1 {\mathbf{E}}_1 +   {\Xi}_2 {\mathbf{E}}_2+   {\Xi}_3 {\mathbf{E}}_3$
  are given by 
\begin{eqnarray*}
  {\Xi}_{0}(u, v) &=&  \quad \; \; \, \left( \alpha \sin v + \beta \cos v \right)  {\mathbf{Ch}}(u), \\
  {\Xi}_{1}(u, v) &=&   - \frac{ \sqrt{ {\alpha}^{2}+{\beta}^{2}+1} }{ \sqrt{ {\alpha}^{2} +1}  }    \sin v  \;  {\mathbf{Sh}}(u), \\
  {\Xi}_{2}(u, v) &=&   - \frac{ \sqrt{ {\alpha}^{2}+{\beta}^{2}+1} }{ \sqrt{ {\beta}^{2} +1}  }    \cos v  \;  {\mathbf{Sh}}(u), \\
  {\Xi}_{3}(u, v) &=&  v.
\end{eqnarray*}
 When the deformation constant $c = \alpha + i \beta \in \mathbb{R} + i  \mathbb{R}$ is equal to zero, the surface ${\Sigma}_{c=0}$  given by the patch ${\mathbf{X}}^{c=0}$ recovers the helicoid. Now now on, consider the case $\left(\alpha, \beta\right) \neq (0,0)$. Fixing the last coordinate $v=v_0$ on $\Sigma$, we examine the level set ${\mathcal{C}}_{v_0} := {\Sigma}^{c} \cap \left\{ {\Xi}_{3}=v_{0} \right\} $ given by  
\[
{\Xi}_0 \left(u, v_0 \right) {\mathbf{E}}_{0} +   {\Xi}_1 \left(u, v_0 \right) {\mathbf{E}}_{1} +   {\Xi}_2 \left(u, v_0 \right) {\mathbf{E}}_{2} + v_0 
{\mathbf{E}}_{3}.
\]
Translating the level set ${\mathcal{C}}_{v_0}$ in the $ {\mathbf{E}}_{0}$ direction yields the curve given by 
\[
{\mathbf{c}}(u) = {\Xi}_0 \left(u, v_0 \right) {\mathbf{E}}_{0} +   {\Xi}_1 \left(u, v_0 \right) {\mathbf{E}}_{1} +   {\Xi}_2 \left(u, v_0 \right) {\mathbf{E}}_{2}.  
\] 
We introduce the new orthonormal frame 
\begin{eqnarray*}
   {\epsilon}_0 &=&  {\mathbf{E}}_{0}, \\
   {\epsilon}_1 &=&  {\epsilon}_1 \left( v_0 \right) =  \frac{1}{ \sqrt{ \frac{  \sin^{2} v_0  }{   {\alpha}^{2} +1   }  +       \frac{  \cos^{2} v_0  }{   {\beta}^{2} +1   } } } \left(  \frac{   \sin v_0 }{ \sqrt{ {\alpha}^{2} +1}  }   {\mathbf{E}}_1 +  \frac{  \cos v_0  }{ \sqrt{ {\beta}^{2} +1}  }       {\mathbf{E}}_2 \right),  \\
  {\epsilon}_2  &=&   {\epsilon}_2 \left( v_0 \right) =  \frac{1}{ \sqrt{ \frac{  \sin^{2} v_0 }{   {\alpha}^{2} +1   }  +       \frac{  \cos^{2} v_0  }{   {\beta}^{2} +1   } } } \left(   \frac{  \cos v_0  }{ \sqrt{ {\beta}^{2} +1}  }   {\mathbf{E}}_1 - \frac{   \sin v_0 }{ \sqrt{ {\alpha}^{2} +1}  }   {\mathbf{E}}_2 \right).  
\end{eqnarray*}
to rewrite ${\mathbf{c}}(u) = \mathbf{x}(u) {\epsilon}_0 +  \mathbf{y}(u)   {\epsilon}_1 +  \mathbf{z}(u) {\epsilon}_{2}$ with components 
\begin{eqnarray*}
\mathbf{x}  &=& {\mathbf{c}}(u)  \cdot {\epsilon}_0   =   \left( \alpha \sin v_0 + \beta \cos v_0 \right) \, {\mathbf{Ch}}(u),   \\
\mathbf{y}  &=&   {\mathbf{c}}(u)  \cdot {\epsilon}_1 =   - \sqrt{  \left( {\alpha}^{2}+{\beta}^{2}+1 \right) \left(   \frac{  \sin^{2} v_0  }{   {\alpha}^{2} +1   }  +       \frac{  \cos^{2} v_0  }{   {\beta}^{2} +1   }       \right) } \, {\mathbf{Sh}}(u),    \\
\mathbf{z}  &=&  {\mathbf{c}}(u)  \cdot {\epsilon}_2  =0,   
\end{eqnarray*}   
 We distinguish two cases:
\begin{enumerate}
\item When $\alpha \sin v_0 + \beta \cos v_0  \neq 0$, the level set  ${\mathcal{C}}_{v_0}$ is congruent to the hyperbola
\[
   {\left(  \frac{  \mathbf{x}  }{   \alpha \sin v_0 + \beta \cos v_0  }    \right)}^{2} -   {\left(    \frac{\mathbf{y}}{  \sqrt{  \left( {\alpha}^{2}+{\beta}^{2}+1 \right) \left(   \frac{  \sin^{2} v_0  }{   {\alpha}^{2} +1   }  +       \frac{  \cos^{2} v_0  }{   {\beta}^{2} +1   }       \right) }  }  \right)}^{2} = 1,
\]
 which has two orthogonal asymptotic lines $\mathbf{x} = \pm \frac{ \, \alpha \sin v_0 + \beta \cos v_0 \, } 
 {      \sqrt{   \frac{  \sin^{2} v_0  }{   {\alpha}^{2} +1   }  +       \frac{  \cos^{2} v_0  }{   {\beta}^{2} +1   } \,   }   } \, \mathbf{y}$.
\item When $\alpha \sin v_0 + \beta \cos v_0 =0$,  we have ${\mathbf{x}}_{0}  \equiv 0$ and ${\mathbf{z}}_{0}  \equiv 0$. This means that the level set curve ${\mathcal{C}}_{v_0}$ becomes 
a line.
\end{enumerate} 
\end{example}

\begin{example}[\textbf{From catenoids in ${\mathbb{R}}^{3}$ to the Hoffman-Osserman minimal surfaces in ${\mathbb{R}}^{4}$ foliated by conic sections with eccentricity smaller than $1$: ellipses or circles}] \label{deformation of catenoids} The catenoid of the neck radius $1$ is, up to translations, given by
\[
  {\mathbf{X}}\left( \zeta \right)   
 =  \textrm{Re}  \left(    \int  \frac{1}{2}  \left(1 -G^2 \right)  \Psi d\zeta, \,  \int \frac{i}{2} \left(1 +G^2 \right)  \Psi d\zeta, \,  \int  G  \Psi d\zeta\right),
\]
with the Weierstrass data 
$\left(G(\zeta), \Psi(\zeta)   d\zeta\right)=\left(  \zeta, \frac{1}{{\zeta}^2}  d\zeta \right),  \; \text{where} \; \zeta =u+iv \in \mathbb{C}-\{0\}$.
 Let $c \in \mathbb{C}$ be the deformation constant. Applying Theorem \ref{thm:main} to the Weierstrass data $\left(G(\zeta), \Psi(\zeta)   d\zeta\right)$, we obtain the minimal surface ${\Sigma}^{c}$ in ${\mathbb{R}}^{4}$ with the patch
\[
{\mathbf{X}}^{c}\left( \zeta \right)    =   \textrm{Re} \left( \int  {\widehat{\phi}}_{0} \, d\zeta, \,   \int {\widehat{\phi}}_{1}  \, d\zeta, \,   \int 
{\widehat{\phi}}_{2} \, d\zeta, \,  \int {\widehat{\phi}}_{3} \, d\zeta \right).
\]  
Here, the holomorphic null curve is given by
\begin{equation}  \label{HO curve}
 \begin{bmatrix} 
 {\widehat{\phi}}_{0} \\  {\widehat{\phi}}_{1} \\   {\widehat{\phi}}_{2} \\  {\widehat{\phi}}_{3}
 \end{bmatrix}  
 = 
    \begin{bmatrix} 
     c \, G^2  \Psi      \\    \frac{1}{2}  \left(1 + \left(c^2 -1\right)G^2 \right)  \Psi      \\  
    \frac{i}{2} \left(1 +\left(c^2 +1\right)G^2 \right)  \Psi      \\    G  \Psi  
 \end{bmatrix} 
 = 
    \begin{bmatrix} 
     c       \\    \frac{1}{2}  \left(   \frac{1}{z^2} + c^2 -1 \right) \\  
    \frac{i}{2} \left(   \frac{1}{z^2} + c^2 +1 \right)     \\   \frac{1}{z} 
 \end{bmatrix}, 
\end{equation}
which recovers the holomorphic data \cite[Proposition 6.6]{HO80} of the Hoffman-Osserman minimal surfaces foliated by 
ellipses or circles \cite[Remark 1]{HO80}. The conformal harmonic immersion 
of the Hoffman-Osserman minimal surfaces in ${\mathbb{R}}^{5}$ should read 
\[
\mathbf{X} = \mathrm{Re} \left(   d_1 \zeta  - \frac{C}{\zeta}, \, d_2 \zeta  - i \frac{C}{\zeta},  \, \alpha \log \zeta, \,  d_4 z,  \, d_5 z \right).
\]
Here, $d_1$, $d_2$, $C$, $d_4$, $d_5$ are complex constants and $\alpha$ is a positive real constant satisfying 
\[
 \left( d_{1}, d_{2} \right) =  \left(  \frac{C}{{\alpha}^{2}}     {\left(   {d_4}^{2}  + {d_5}^{2}   \right)}  - \frac{{\alpha}^{2}}{4C} ,    
        i    \left( \frac{C}{{\alpha}^{2} }    {\left(   {d_4}^{2}  + {d_5}^{2}   \right)}   + \frac{{\alpha}^{2}}{4C}    \right)   \right),
\]
which guarantees that the induced holomorphic null curve in ${\mathbb{C}}^{5}:$
\begin{equation} \label{HO original}
\left(   d_1  + \frac{C}{{\zeta}^{2}}, \, d_2  + i \frac{C}{{\zeta}^{2}},  \, \frac{\alpha}{\zeta}, \,  d_{4},  \, d_{5} \right).
\end{equation}
Taking the normalization $\left( d_{1}, d_{2}, C, d_{4}, d_{5}, \alpha \right) = \left( \frac{1}{2}  ( c^2 -1 ), \frac{i }{2}  ( c^2 +1 ), \frac{1}{2}, c, 0, 1 \right)$ in 
 the Hoffman-Osserman curve (\ref{HO original}), we recover the holomorphic curve, which is equivalent to (\ref{HO curve}).  
\end{example}

\begin{theorem}[\textbf{Minimal surfaces in ${\mathbb{R}}^{4}$ spanned by circles at infinity}] \label{twocircles} 
Given a constant $\theta \in \left(0, \frac{\pi}{2}\right)$, we define the minimal surface ${\Sigma}^{\tan \theta}$ defined by the conformal harmonic mapping 
\[
  {\mathbf{X}}^{\tan \theta} (U, V) = \left( \frac{\sin \theta}{\cos \theta}  \sinh U \cos V,  \cosh U \cos V,  \frac{1}{\cos \theta}  \cosh U \sin V, U \right), \;
  \left(U, V\right) \in {\mathbb{R}}^{2}. 
\]
Then, we have the following properties.
\begin{enumerate}
\item  For each constant height ${\mathbf{x}}_{4}=U_0$, the level set  ${\mathcal{C}}_{U_0}= {\Sigma}^{\tan \theta} \cap \{{\mathbf{x}}_{4}=U_0\}$ is an ellipse. In particular, the neck ${\mathcal{C}}_{0}={\Sigma}^{\tan \theta} \cap \{{\mathbf{x}}_{4}=0\}$ is congruent to $   {\mathbf{x}}^{2}   + \left(  {\cos}^{2} \theta  \right) \,   {\mathbf{y}}^{2} =1$. 
\item When $U$ approaches to $\infty$ (or $-\infty$), the ellipse ${\mathcal{C}}_{U}$ converges to a circle. 
\end{enumerate}
\end{theorem}

\begin{proof}
We recall the classical Weierstrass data of the catenoid  in ${\mathbb{R}}^{3}$ with neck size $1$: 
\[
\left(z,  \frac{1}{z^2} dz \right) = \left( e^{\zeta}, e^{-\zeta} d\zeta \right) =: \left(G(\zeta), \Psi(\zeta)   d\zeta\right).
\]
Applying Corollary \ref{main:thm2}, we have the minimal surface in ${\mathbb{R}}^{4}$, up to translations, given by
\[
\left( {\mathbf{x}}_{0}, {\mathbf{x}}_{1}, {\mathbf{x}}_{2}, {\mathbf{x}}_{3} \right)  
 = \left( \textrm{Re} \int   {\widetilde{\,\phi\,}}_{0}\left( \zeta \right) d\zeta, \,    \textrm{Re} \int   {\widetilde{\,\phi\,}}_{1}\left( \zeta \right) d\zeta, \, \textrm{Re} \int   {\widetilde{\,\phi\,}}_{2}\left( \zeta \right) d\zeta, \, \textrm{Re} \int   {\widetilde{\,\phi\,}}_{3}\left( \zeta \right) d\zeta \right),
\]
  where the holomorphic curve   $ {\widetilde{\,\phi \,}} = \left( {\widetilde{\,\phi\,}}_{0}, \, {\widetilde{\,\phi\,}}_{1}, \,  {\widetilde{\,\phi\,}}_{2}, {\widehat{\,\phi\,}}_{3}\right)$ is  determined by 
\[
   {\widetilde{\,\phi \,}}
= \left(   \frac{\sin \theta}{2} \left(e^{-\zeta} + \frac{ e^{\zeta}}{{\cos}^{2} \theta} \right), \,  \frac{\cos \theta}{2}  \left(e^{-\zeta} - \frac{ e^{\zeta}}{{\cos}^{2} \theta} \right)  \Psi  , \,  \frac{i}{2} \left(e^{-\zeta} + \frac{ e^{\zeta}}{{\cos}^{2} \theta} \right)   \Psi, \,   1   \right). 
\]
We write $\zeta=u+iv$ and introduce the  coordinates $(U, V):=(u - \ln \left( \cos \theta \right), v)$.
Applying  reflection $\left( {\mathbf{x}}_{0}, {\mathbf{x}}_{1}, {\mathbf{x}}_{2}, {\mathbf{x}}_{3} \right)  \to 
\left( {\mathbf{x}}_{0}, - {\mathbf{x}}_{1}, -{\mathbf{x}}_{2}, {\mathbf{x}}_{3} \right)  $ and translation ${\mathbf{x}}_{3} \to {\mathbf{x}}_{3} - \ln \left( \cos \theta \right)$, the above patch recovers the desired conformal harmonic mapping ${\mathbf{X}}^{\tan \theta} (U, V)$. 
One can easily check that the level set ${\mathcal{C}}_{U_0}= {\Sigma}^{\tan \theta} \cap \{{\mathbf{x}}_{4}=U_0\}$ is congruent to the ellipse
\[
    {\left(  \frac{  \mathbf{x} }{ r_{1} }  \right)}^{2} +    {\left(  \frac{  \mathbf{y} }{ r_{2} }  \right)}^{2} =1, \quad \left( r_{1}, r_{2} \right) = \left(  \sqrt{   
    {\left( \frac{\sin \theta}{\cos \theta}  \sinh U \right) }^{2} +{\cosh}^{2} U \,  }, \,  \frac{1}{\cosh \theta}  \cosh U    \right).
\]
When $\vert U \vert \to \infty$ (or  $\,\vert \tanh U \vert \to 1$), the ellipse ${\mathcal{C}}_{U}$ should converge to a circle:
\[
\lim_{\vert U \vert  \to \infty} \frac{{r_{2}}^{2} }{{r_{1}}^{2}} = \lim_{\vert U \vert \to \infty} \frac{  \frac{ 1}{ {\cos}^{2} \theta}   }{    
    {\left( \frac{\sin \theta}{\cos \theta}  \tanh U \right) }^{2} +   1       }  = \frac{  \frac{ 1}{ {\cos}^{2} \theta}  }{  {\tan}^{2} \theta +1} =1.
\]
\end{proof}

\begin{example}[\textbf{Lagrangian catenoid in ${\mathbb{R}}^{4}$}] \label{Lcat} Rotating and dilating the Lagrangian catenoid
\[
\left\{ \left({\zeta}, \frac{1}{\zeta} \right) \in {\mathbb{C}}^{2}  \; \vert \;   \zeta \in  {\mathbb{C}} - \{0\} \; \right\} = 
\left\{ \left(e^{w}, e^{-w} \right) \in {\mathbb{C}}^{2}  \; \vert \;   w \in  {\mathbb{C}}   \right\}
\]  
we have the holomorphic curve in ${\mathbb{C}}^{2}$:
\[
 \left\{ \frac{1}{\sqrt{2}} \left(\frac{1}{\sqrt{2}} e^{w} -  \frac{1}{\sqrt{2}} e^{-w}, \frac{1}{\sqrt{2}} e^{w} + \frac{1}{\sqrt{2}} e^{-w}  \right)
= \left(  \sinh w, \cosh w \right) \in {\mathbb{C}}^{2}  \; \vert \;   w \in  {\mathbb{C}}  \right\},
\]  
which can be identified as a minimal surface in ${\mathbb{R}}^{4}$, with coordinates $(u, v)=\left(   \textrm{Re} \, w,  \textrm{Im} \, w   \right)$, 
\[
\Sigma 
=\left\{  {\mathbf{X}}(u, v) = \left(  \sinh u \cos v, \cosh u \sin v, \cosh u \cos v, \sinh u \sin v \right) \in {\mathbb{R}}^{4}  \; \vert \;   (u, v) \in  {\mathbb{R}}^{2}   \; \right\}.
\]  
We obtain two different families of planar curves on $\Sigma$: 
\begin{enumerate}
\item  Fixing $v=v_{0}$, the curve $u \mapsto {\mathbf{X}}(u, v_0)$ can be viewed as  
\[
u \mapsto {\mathbf{X}}(u, v_0) = \cosh u \left( 0,   \sin v_{0},   \cos v_{0}, 0 \right) + \sinh u \left(  \cos v_{0}, 0,  0,  \sin v_{0}\right),
\] 
which is congruent to the hyperbola ${\mathbf{x}}^{2} - {\mathbf{y}}^{2}=1$, as unit vectors $\left( 0,   \sin v_{0},   \cos v_{0}, 0 \right)$ and 
$ \left(  \cos v_{0}, 0,  0,  \sin v_{0}\right)$ are orthogonal to each other.
\item Fixing $u=u_{0}$, the curve $v \mapsto {\mathbf{X}}(u_{0}, v)$ can be viewed as 
\[
v \mapsto  {\mathbf{X}}(u, v_0) = \cos v \left(  \sinh u_{0}, 0,   \cosh u_{0}, 0 \right) + \sinh u \left( 0,  \cosh u_{0},  0,  \sinh u_{0} \right),
\] 
which is congruent to the circle ${\mathbf{x}}^{2} + {\mathbf{y}}^{2}={\cosh}^{2} u_{0}  +{\sinh}^{2} u_{0}$, as two orthogonal  vectors $\left(  \sinh u_{0}, 0,   \cosh u_{0}, 0 \right) $ and $ \left( 0,  \cosh u_{0},  0,  \sinh u_{0} \right)$ have the same length $\sqrt{{\cosh}^{2} u_{0}  +{\sinh}^{2} u_{0}}$.
\end{enumerate}
More generally, one can easily check that, given constants ${\lambda}_{1}, {\lambda}_{2}, {\lambda}_{3}, {\lambda}_{4} \in \mathbb{R}$,  the holomorphic 
curve $\left\{ \left(  {\lambda}_{1}  e^{w} + {\lambda}_{2}  e^{-w},  {\lambda}_{3}  e^{w} + {\lambda}_{4}  e^{-w} \right) \in {\mathbb{C}}^{2}  \; \vert \;   w \in  {\mathbb{C}} \right\}$ admit similar properties. 
\end{example}

\begin{example}[\textbf{Minimal surfaces in ${\mathbb{R}}^{4}$ foliated by conic sections with eccentricity $1$: parabolas}] \label{complex parabola} Let $\mu \in \mathbb{C}-\{0\}$ be a constant. We shall see that the holomorphic curve $z= \mu w^2$ in ${\mathbb{C}}^{2}$ becomes a minimal surface ${\Sigma}^{\mu}$ in ${\mathbb{R}}^{4}$ foliated by a family of parabolas $\mathbf{y}= \lambda   \mathbf{x}^2$ with $\lambda \in \left(0, \vert \mu \vert \right)$.  
With the frame ${\mathbf{E}}_{1}=(1,0,0,0)$, ${\mathbf{E}}_{2}=(0,1,0,0)$, ${\mathbf{E}}_{3}=(0,0,1,0)$, ${\mathbf{E}}_{4}=(0,0,0,1)$, we prepare the  immersion ${\mathbf{X}}^{\mu}$ of the complex parabola  ${\Sigma}^{\mu}$, with the conformal coordinate  $\zeta= u+iv \in \mathbb{C}:$
\[
  {\mathbf{X}}^{\mu} (u, v) =  x_1 {\mathbf{E}}_{1} + x_2 {\mathbf{E}}_{2} + x_3 {\mathbf{E}}_{3} x_3 +  x_4 {\mathbf{E}}_{4}= \textrm{Re} \,  {\zeta}  \,  {\mathbf{E}}_{1} +   \textrm{Im}\,  {\zeta}  \, {\mathbf{E}}_{2} +  \textrm{Re} \, \left(  \mu {\zeta}^{2} \right)  {\mathbf{E}}_{3}+  \textrm{Im} \, \left(  \mu {\zeta}^{2} \right)  {\mathbf{E}}_{4}.
  \]
  Fix $u=u_0$. Let's look at the slice ${\mathcal{C}}_{u_0} = {\Sigma}^{\mu} \cap \{x_1 = u_{0}\}$ given by $\mathbf{c}(v)={\mathbf{X}}^{\mu} (u_0, v)$. 
  Writing $(a, b)=\left(  \textrm{Re} \, \mu,  \textrm{Im} \, \mu   \right) \neq (0, 0)$ and introducing the new orthonormal frame
\begin{eqnarray*}
  {\mathbf{e}}_1  &=&   {\mathbf{E}}_{1} , \\
  {\mathbf{e}}_2  &=&    {\mathbf{e}}_2 (u_{0})  =  \frac{1 }{\sqrt{ 1 + 4 \left( {a}^{2}+{b}^{2}  \right) {u_{0}}^{2}  \,     }}  \left( {\mathbf{E}}_{2} -2b {\mathbf{E}}_{3} +2a {\mathbf{E}}_{4} \right) ,  \\
   {\mathbf{e}}_3   &=&     {\mathbf{e}}_3 (u_{0})  =  \frac{1 }{\sqrt{  {a}^{2}+{b}^{2}  }}  \left(-a  {\mathbf{E}}_{3} - b {\mathbf{E}}_{4}  \right) ,  \\
    {\mathbf{e}}_4   &=&  {\mathbf{e}}_4 (u_{0})  =  \frac{1 }{\sqrt{   \left( {a}^{2}+{b}^{2}  \right) \left( 1 + 4 {u_{0}}^{2}  \right) \,     }}  \left( {\mathbf{E}}_{2} -2b {\mathbf{E}}_{3} +2a {\mathbf{E}}_{4} \right),  \\
\end{eqnarray*}
we rewrite the patch of the level set curve ${\mathcal{C}}_{u_0}$:
\[
\mathbf{c}(v)=  u_0 {\mathbf{e}}_{1}  + v \sqrt{ 1 + 4 \left( {a}^{2}+{b}^{2}  \right) {u_{0}}^{2}  \,     }  \,  {\mathbf{e}}_2
+ \left( v^2  - {u_{0}}^{2} \right) \sqrt{  {a}^{2}+{b}^{2}  } \, {\mathbf{e}}_{3}.
\]
This indicates that the slice ${\mathcal{C}}_{u_0}$ is congruent to the planar curve 
\[
  (\mathbf{x}(v), \mathbf{y}(v)) = (v \sqrt{ 1 + 4 \left( {a}^{2}+{b}^{2}  \right) {u_{0}}^{2}  \,     } , \left( v^2  - {u_{0}}^{2} \right) \sqrt{  {a}^{2}+{b}^{2}  }  \, ).
\]
This curve represents the parabola
$ \mathbf{y} = \frac{ \sqrt{  {a}^{2}+{b}^{2}  }  }{   1 + 4 \left( {a}^{2}+{b}^{2}  \right) {u_{0}}^{2}     } { \mathbf{x}}^2 - 
 {u_{0}}^{2}\sqrt{  {a}^{2}+{b}^{2}  }$.
\end{example}

\begin{remark}
D. Joyce \cite{Joyce2001} constructed special Lagrangian submanifolds in ${\mathbb{R}}^{2n}$ evolving quadrics. In the case when $n=2$, those examples 
become holomorphic curves in  ${\mathbb{C}}^{2}$.
\end{remark}

\begin{remark}
 Motivated by M. Shiffman's theorems  for minimal surfaces in ${\mathbb{R}}^{3}$ bounded by two convex curves  \cite{Sh56}, F. L\'{o}pez, R. L\'{o}pez, and R. Souam  \cite{LLS2000} generalized  Riemann's minimal surfaces  \cite{MP2015, Riem1868} in ${\mathbb{R}}^{3}$ foliated by circles and lines  to maximal surfaces in Lorentz-Minkowski space ${\mathbb{L}}^{3}=\left( {\mathbb{R}}^{3}, \, d{x_{1}}^2 +d{x_{2}}^2 -d{x_{3}}^2 \right)$ foliated by conic sections.   Unlike Euclidean space,  since Lorentz-Minkowski space ${\mathbb{L}}^{3}$ admits three different rotational isometries (elliptic, hyperbolic, parabolic rotations), there exist fruitful examples of maximal surfaces foliated by conic sections. 
 \end{remark}

\section*{Acknowledgements}

The author would like to thank Marc Soret and Marina Ville for enlightening discussion. Part of this work was carried out while he was visiting Universit\'e Francois Rabelais in December, 2016. He would like to warmly thank M. S. and M. V. for their hospitality and support.

\bigskip


\begin{thebibliography}{00} 

\bibitem{BB2011}
J. Bernstein, C. Breiner, \textit{Symmetry of embedded genus 1 helicoids}, Duke Math. J. \textbf{159} (2011), no. 1, 83--97. 

\bibitem{Cal53} E. Calabi, \textit{Isometric imbeddings of complex manifolds}, Ann. of Math. \textbf{58} (1953), 1--23.

\bibitem{CU1999} I. Castro, F. Urbano, \textit{On a minimal Lagrangian submanifold of  ${\mathbb{C}}^{n}$ foliated by spheres}, 
Michigan Math. J. \textbf{46} (1999), no. 1, 71--82.

\bibitem{Deu2013}
M. Deutsch, \textit{Integrable deformation of critical surfaces in spaceforms}, Bull. Braz. Math. Soc. (N.S.) \textbf{44} (2013), no. 1, 1--23. 

\bibitem{DHS2010}
U. Dierkes, S. Hildebrandt, F. Sauvigny, \textit{Minimal Surfaces}, Grundlehren der mathematischen Wissenschaften, vol. 339. Springer, Berlin (2010)

 \bibitem{DT09} 
M.  Dajczer, R. Tojeiro, \textit{All superconformal surfaces in ${\mathbb{R}}^{4}$ in terms of minimal surfaces}, Math. Z. \textbf{261} (2009), no. 4, 869--890. 
 
\bibitem{G1887}
E. Goursat,  \textit{Sur un mode de transformation des surfaces minima}, Acta Math. \textbf{11} (1887), 135--186.

\bibitem{HL1982}
R. Harvey, H. B. Lawson, \textit{Calibrated geometries}, Acta Math. \textbf{148} (1982), 47--157.

\bibitem{Joyce2001}
D. Joyce, \textit{Constructing special Lagrangian m-folds in  ${\mathbb{C}}^{m}$ by evolving quadrics}, Math. Ann. 
\textbf{320} (2001), no. 4, 757--797. 
 
\bibitem{LK2016}
K. Leschke, K. Moriya, \textit{Applications of quaternionic holomorphic geometry to minimal surfaces}, Complex Manifolds \textbf{3} (2016), no. 1, 
282--300

\bibitem{Law71} H. B. Lawson, \textit{Some intrinsic characterizations of minimal surfaces}, J. Analyse Math. \textbf{24} (1971), 151--161. 

\bibitem{LLS2000}
F. J. L\'{o}pez, R. L\'{o}pez, R. Souam, \textit{Maximal surfaces of Riemann type in Lorentz-Minkowski space ${\mathbb{L}}^{3}$}, Mich. J. Math. 
\textbf{47} (2000), 469--497.

\bibitem{LR91} F. J. L\'{o}pez, A. Ros, \textit{On embedded complete minimal surfaces of genus zero}, J. Differential Geom. \textbf{33} (1991), 293--300.  

\bibitem{MB2000} I. Mladenov, A. Borislav, \textit{Deformations of minimal surfaces}, Geometry, Integrability and Quantization (Varna, 1999). Coral Press Sci. Publ., Sofia (2000).

\bibitem{MeRo05} W. H. Meeks III, H. Rosenberg, \textit{The uniqueness of the helicoid}, Ann. of Math. (2) \textbf{161} (2005), no. 2, 727--758. 

  \bibitem{Mor09}
K. Moriya, \textit{Super-conformal surfaces associated with null complex holomorphic curves}, Bull. Lond. Math. Soc. \textbf{41} (2009), no. 2, 327--331.
 
\bibitem{MM2015} A. Moroianu, S.  Moroianu, \textit{Ricci surfaces}, Ann. Sc. Norm. Super. Pisa Cl. Sci. (5) \textbf{14} (2015), no. 4, 1093--1118. 

\bibitem{MP03} P. Mira, J. A. Pastor, \textit{Helicoidal maximal surfaces in Lorentz-Minkowski space}, 
Monatsh. Math. \textbf{140} (2003), no. 4, 315--334.

 \bibitem{MP2012} 
W. H. Meeks III, J. P\'{e}rez, \textit{A Survey on Classical Minimal Surface Theory}, University Lecture Series, vol. 60. AMS, Providence (2012).
 
\bibitem{MP2015}
W. H. Meeks III, J. P\'{e}rez, \textit{The Riemann minimal examples}, The legacy of Bernhard Riemann after one hundred and fifty years, Advanced Lectures in Mathematics, \textbf{35} (2015), 417--457.

\bibitem{HO80} D. A. Hoffman, R. Osserman, \textit{The geometry of the generalized Gauss map},
Mem. Amer. Math. Soc. \textbf{28} (1980), no. 236. 

\bibitem{Oss86} R. Osserman, \textit{A survey of minimal surfaces}. Second edition. Dover Publications, Inc., New York, 1986. 
 
 \bibitem{Park2015}
 S.-H. Park, \textit{Circle-foliated minimal surfaces in $4$-dimensional space forms}, Bull. Korean Math. Soc. 
\textbf{52} (2015), no. 5, 1433--1443.

 \bibitem{PR1993}  J. P\'{e}rez, A. Ros,
\textit{Some uniqueness and nonexistence theorems for embedded minimal surfaces}, Math. Ann. \textbf{295} (1993), 513--525.
 
\bibitem{PR2002}  J. P\'{e}rez, A. Ros, \textit{Properly embedded minimal surfaces with finite total curvature}, The Global Theory of Minimal Surfaces in Flat Spaces (Martina Franca, 1999), Lecture Notes in Math., vol. 1775, Springer-Verlag, Berlin, 2002, pp. 15--66.

 \bibitem{Riem1868} 
  B. Riemann, \textit{\"{U}ber die Fl\"{a}chen vom Kleinsten Inhalt be gegebener Begrenzung}, Abh. K\"{o}nigl. Ges. Wiss. G\"{o}ttingen, 
  Math. Kl. \textbf{13} (1868), 329--333.
  
  \bibitem{Romon1997}
  P. Romon, \textit{Symmetries and conserved quantities for minimal surfaces}, preprint (1997).
  
  \bibitem{Ros1996}
  A. Ros, \textit{Embedded minimal surfaces: forces, topology and symmetries}, Calc. Var. Partial Differential Equations, \textbf{4} (1996), 469--496.
  
\bibitem{Sch83} R. Schoen, \textit{Uniqueness, symmetry, and embeddedness of minimal surfaces}, J. Differential Geom. \textbf{18} (1983), 791--809. 

 
\bibitem{Sh56}
M. Shiffman, \textit{On surfaces of stationary area bounded by two circles, or convex curves, in parallel planes}, 
Ann. of Math. \textbf{63} (1956), 77--90.


\end{thebibliography}
\end{document}